\documentclass[12pt,leqno]{amsart}
\usepackage{a4wide}
\usepackage{amssymb,amsmath,amscd,graphicx,fontenc,amsthm,mathrsfs, mathtools}
\usepackage[frame, cmtip, arrow, matrix, line, graph, curve]{xy}
\usepackage{graphpap, color}
\usepackage[mathscr]{eucal}
\usepackage{ifthen}
\usepackage{relsize}
\usepackage{bm} 
\usepackage[normalem]{ulem}
\usepackage{cancel}

\usepackage{hyperref}

\usepackage{fancyhdr}

\usepackage{indentfirst}
\usepackage{paralist}
\usepackage{pstricks}
\usepackage{bbm}
\usepackage{tikz}
\usepackage{cleveref}

\newtheorem{theorem}{Theorem}[section]

\newtheorem{lemma}[theorem]{Lemma}
\newtheorem{proposition}[theorem]{Proposition}

\newtheorem*{theorem*}{Theorem}

\numberwithin{equation}{section}

\newcommand{\NN}{{\mathbb{N}}}
\newcommand{\ZZ}{{\mathbb{Z}}}

\newcommand{\HH}{{\mathbb{H}}}
\newcommand{\RR}{{\mathbb{R}}}

\DeclareMathOperator{\vol}{vol}

\newcommand{\diag}{\mathrm{diag}}

\renewcommand{\mod}{\mathrm{mod}\;}

\newcommand{\SL}{\mathrm{SL}}

\newcommand{\PGL}{\mathrm{PGL}}

\newcommand{\Id}{\mathrm{Id}}

\renewcommand{\bf}[1]{{\mathbf{#1}}}

\newcommand{\Siegel}[2]{\widehat{#1}^{(#2)}}

\makeatletter
\def\moverlay{\mathpalette\mov@rlay}
\def\mov@rlay#1#2{\leavevmode\vtop{%
   \baselineskip\z@skip \lineskiplimit-\maxdimen
   \ialign{\hfil$\m@th#1##$\hfil\cr#2\crcr}}}
\newcommand{\charfusion}[3][\mathord]{
    #1{\ifx#1\mathop\vphantom{#2}\fi
        \mathpalette\mov@rlay{#2\cr#3}
      }
    \ifx#1\mathop\expandafter\displaylimits\fi}
\makeatother

\hypersetup{
    colorlinks=true,
    linkcolor=blue,
    filecolor=magenta,      
    urlcolor=magenta,
    citecolor=blue,
}

\definecolor{cmd}{rgb}{1.0, 0.35, 0.21}

\begin{document}

\title[Moments of Siegel transforms with congruence cond. in dim 2]{Moment formulas of Siegel transforms with congruence conditions in dimension 2}
\author{Jiyoung Han}
\address{Department of Mathematics Education, Pusan National University, Busan 46241, Republic of Korea}
\email{jiyoung.han@pusan.ac.kr}
\author{Seulbee Lee}
\address{Department of Mathematical Sciences, Seoul National University, 1 Gwanak-ro, Gwanak-gu, Seoul 08826, Republic of Korea}
\email{seulbee.lee@snu.ac.kr}



\begin{abstract} 
We compute the first and second moment formulas for Siegel transforms related to problems counting primitive lattice points in the real plane with congruence conditions. As applications, we derive an analog of Schmidt's random counting theorem and the quantitative Khintchine theorem for irrational numbers, approximated by rational numbers $p/q$, where we place a congruence-conditional constraint on the vector $(p,q)$. 
\end{abstract}

\maketitle

\section{Introduction}\label{Sect: Introduction}
The Siegel transform is the map sending a bounded function $f$ of compact support on $\RR^d$ to an integrable function on the homogeneous space $\SL_d(\RR)/\SL_d(\ZZ)$
 defined as
\[
\widetilde{f}(g\SL_d(\ZZ))=\sum_{\bf{v}\in \ZZ^d-\{\bf{0}\}} f(g\bf{v}),
\quad \forall g\SL_d(\ZZ)\in \SL_d(\RR)/\SL_d(\ZZ).
\]
The Siegel transform, together with the Siegel integral formula~\cite{Siegel45} ($d\ge 2$)
and Rogers' second moment formula \cite{Rogers55} ($d\ge 3$) play a fundamental role in the applications of homogeneous dynamics to problems in the geometry of numbers, which are related to counting the number of lattice points in certain conditions \cite{Schmidt1960, AM09, AM18, KY18, KY2019, BGH2021, Fairchild2021, KS2019, KY2023, KS}.

The primitive Siegel transform, which is particularly more effective than the standard Siegel transform in the case $d=2$, is defined as follows. 
\[
\widehat{f}(g\SL_d(\ZZ))=\sum_{\bf{v}\in P(\ZZ^d)} f(g\bf{v}),
\]
where $P(\ZZ^d)=\{\bf{v}\in \ZZ^d: \gcd\bf{v}=1\}$. It is well-known that the primitive Siegel transform $\widehat{f}$ is a bounded function thus $\widehat{f}\in L^k(\SL_2(\RR)/\SL_2(\ZZ))$ for any $k\ge 1$, whereas $\widetilde{f}$ is not in $L^2(\SL_2(\RR)/\SL_2(\ZZ))$ in general. 
Moreover, for applications relative to counting primitive vectors in $\RR^2$, it is useful to consider the first and second moment formulas of the variable
\[
\left(t^{1/2},g\:\SL_2(\ZZ)\right) \mapsto t^{1/2} \widehat{f}\left(t^{1/2},g\:\SL_2(\ZZ)\right):=\sum_{\bf{v}\in P(\ZZ^2)} f(t^{1/2}g\bf{v}),
\]
where $(t, g\:\SL_2(\ZZ))\in (0,1]\times \SL_2(\RR)/\SL_2(\ZZ)$ (see the proof of Theorem 2 in \cite{Schmidt1960}). Note that this set is identified with the set of lattices $$\left\{t^{1/2}g\ZZ^2: t\in (0,1]\;\text{and}\; g\:\SL_2(\ZZ)\in \SL_2(\RR)/\SL_2(\ZZ)\right\},$$ which we refer as \emph{the cone of the homogeneous space $\SL_2(\RR)/\SL_2(\ZZ)$}.

Using these formulas, Schmidt computed an asymptotic formula of the number of primitive lattice points for generic lattices in $\RR^d$, contained in the increasing sequence of Borel sets whose volumes diverge to infinity (\cite[Theorem 2]{Schmidt1960}. See also Theorem 1 for the case when $d\ge 3$).


In this article, we are interested in problems that count primitive vectors under the given congruence condition. 
Our first result is an analog of Schmidt's random counting theorem.

\begin{theorem}\label{thm: Analog of Schmidt}
Let $\{A_T\}$
be an increasing family of Borel sets by inclusion and suppose that
\[
V_T:=\vol(A_T)\rightarrow \infty\;\text{as}\;T\rightarrow \infty.
\]
Let $\psi:\RR_{>0}\rightarrow \RR_{>0}$ be a non-decreasing function satisfying $\int \psi(x)^{-1} dx<\infty$.
For a given $N\in \NN$, take an integer vector $\bf{v}_0\in \ZZ^2$ for which $\gcd(\bf{v}_0, N)=1$.

For almost all unimodular lattice $g\mathbb Z^2 \subseteq \RR^2$, where $g$ varies in $\SL_2(\RR)$ with respect to the Haar measure, it follows that
\[
\# \left(g\{\bf{v}\in P(\mathbb Z^2):\bf{v}\equiv \bf{v}_0\;(\mod N)\}\cap A_T\right)=\frac{V_T} {\zeta_N(2)N^2}+O\left(V_T^{1/2}(\log V_T)^2\psi(V_T)^{1/2}\right),
\]
where $\zeta_N(d)$ for $d\in \NN$ is defined as
\[\zeta_N(d)=\prod_{\substack{
p:\; \text{prime,}\\
p\nmid N}} \left(1-\frac 1 {p^d}\right)^{-1}.
\]
\end{theorem}
We remark that when we take the borel sets $A_T$ as growing balls of volume $V$, there is a stronger error $o(V^{1/2})$, and even $O(V^{5/12+\varepsilon})$ for any positive $\varepsilon>0$ if we assume the Riemann hypothesis. See \cite[Theorem 4.3]{BFC}.
In this case, the lattice subgroup $\Gamma$ in their theorem is the conjugate of $\Gamma_1(N)=\{\gamma\in \SL_2(\ZZ): \gamma(1,0)^T\equiv (1,0)^T\ (\mod N) \}$ subject to a (primitive) vector $\bf{v}_0$ in the congruence condition.

Adopting the tactic of Schmidt \cite{Schmidt1960}, we will derive Theorem~\ref{thm: Analog of Schmidt} by establishing the first and second moment formulas for a new Siegel transform defined as follows: let us denote the congruence condition by $\sigma=(\bf{v}_0, N)$, where $N\in \NN$ and $\bf{v}_0\in \ZZ^2$ with $\gcd(\bf{v}_0, N)=1$. For a bounded function $f:\RR^2\rightarrow \RR$ of compact support, the Siegel transform $\widehat{f}^{(\sigma)}$ relative to the congruence condition $\sigma$ is an integral function on $\SL_2(\RR)/\Gamma(N)$ given as
\[
\widehat{f}^{(\sigma)}(g\Gamma(N))=\sum_{\substack{
\bf{v}\in P(\ZZ^2),\\
\bf{v}\equiv \bf{v}_0\;(\mod N)}} f(g\bf{v}),\quad \forall g\Gamma(N)\in \SL_2(\RR)/\Gamma(N),
\]
where $\Gamma(N)=\left\{g\in \SL_2(\ZZ): g\equiv \Id_2 \;(\mod N)\right\}$ is the principal congruence subgroup of level $N$.
Note that $\widehat{f}^{(\sigma)}$ is well-defined since the set $\{\bf{v}\in \ZZ^2: \bf{v} \equiv \bf{v}_0\;(\mod N)\}$ is $\Gamma(N)$-invariant.

\begin{theorem}\label{thm: moment formulas} Let $f:\RR^2\rightarrow \RR$ be a bounded function of compact support. For a congruence condition $\sigma=(\bf{v_0}, N)$ as above, we obtain the integral formulas below.

\noindent (1) (The first moment formula) It holds that
\[
\int_{\SL_2(\RR)/\Gamma(N)} \Siegel{f}{\sigma}(g\Gamma(N)) d\nu_N(g)
=\frac 1 {\zeta_N(2) N^2} \int_{\RR^2} f({\bf v}) d{\bf v}.
\]

Here $\nu_N$ is the $\SL_2(\RR)$-invariant probability measure on $\SL_2(\RR)/\Gamma(N)$.

\noindent (2) (The second moment formula)
Denote by $f\otimes f$ the function on $\RR^2\times \RR^2$ which is given as $f\otimes f(\bf{v}_1, \bf{v}_2)=f(\bf{v}_1)f(\bf{v}_2)$ for any $\bf{v}_1,\bf{v}_2\in \RR^2$. It follows that

\begin{equation}\label{eqn: second moment}\begin{split}
\int_{\SL_2(\RR)/\SL_2(\ZZ)} \widehat{f}^{(\sigma)}(g\Gamma(N))^2 d\nu_N(g)
&=\sum_{n\in N\ZZ-\{0\}} \frac {\varphi(n)} {\zeta_N(2)N^3\varphi(N)}\int_{\SL_2(\RR)} f\otimes f(gJ_n) d\eta(g) \\
&+\frac 1 {\zeta_N(2) N^2} \int_{\RR^2} f(\bf{v})f(\bf{v})+f(\bf{v})f(-\bf{v}) d\bf{v},
\end{split}\end{equation}
where for each $n\in N\ZZ-\{0\}$, $J_n=\left(\begin{array}{cc}
1 & 0 \\
0 & n\end{array}\right)$, $\varphi(\cdot)$ is Euler totient function,
and $\eta$ is the $\SL_2(\RR)$-invariant measure inherited from the Lebesgue measure under the canonical embedding $\SL_2(\RR)\hookrightarrow \RR^2\times \RR^2$. 
\end{theorem}

%
%
%
%
%

The first part of (R.H.S) in \eqref{eqn: second moment} is the summation of integrals over $\SL_2(\RR)$-embedded images in $\RR^2\times \RR^2$, and each integral is obtained by a folding-unfolding argument and change of $\SL_2(\RR)$-invariant measures of $\SL_2(\RR)$.
To derive Theorem~\ref{thm: Analog of Schmidt} from moment formulas, the formula in \eqref{eqn: second moment} is inappropriate. Instead, we will introduce the notion of \emph{the cone} of $\SL_2(\RR)/\Gamma(N)$ and establish the moment formulas of the variable
\[
t^{1/2}\widehat{f}^{(\sigma)} (t^{1/2}g):=\sum_{\substack{\bf{v}\in P(\ZZ^2),\\
\bf{v}\equiv \bf{v}_0\ (\mod N)}} f(t^{1/2}g\bf{v}), 
\]
where $t\in (0,1]$ and $g$ lies in any given fundamental domain of $\SL_2(\RR)/\Gamma(N)$.

\begin{theorem}\label{thm: moment formulas over the cone}
Let $f:\RR^2\rightarrow \RR$ be a bounded function of compact support. For a congruence condition $\sigma=(\bf{v_0}, N)$ as above, we obtain the integral formulas below.

\noindent (1) (The first moment formula) It holds that
\[
\int_0^1\int_{\SL_2(\RR)/\Gamma(N)} t\widehat{f}^{(\sigma)} (t^{1/2}g)d\nu_N dt
=\frac 1 {\zeta_N(2) N^2} \int_{\RR^2} f(\bf{v}) d\bf{v}.
\]

\noindent (2) (The second moment formula) There is a function $\Phi_N(x)$ on $\RR$ whose approximate value is $1/\zeta_N(2)N$ as $|x|\rightarrow \infty$ and
\[\begin{split}
\int_0^1 \int_{\SL_2(\RR)/\Gamma(N)} \left(t^{1/2} \widehat{f}^{(\sigma)} (t^{1/2}g)\right)^2 d\nu_N dt
&=\frac 1 {\zeta_N(2) N^3} \int_{\RR^2\times \RR^2} \Phi_N(\det(\bf{v}_1, \bf{v}_2)) f(\bf{v}_1)f(\bf{v}_2) d\bf{v}_1 d\bf{v}_2\\
&+\frac 1 {\zeta_N(2)N^2} \int_{\RR^2} f(\bf{v})f(\bf{v}) + f(\bf{v})f(-\bf{v}) d\bf{v}.
\end{split}\]
\end{theorem}

Finally, we state the quantitative version of Khintchine theorem having congruence and primitive conditions, as a corollary of Theorem~\ref{thm: Analog of Schmidt}.

\begin{theorem}\label{thm: quantitative K}
Let $\psi:\RR_{>0}\rightarrow\RR_{>0}$ be a non-increasing function such that $\sum_{t\in \NN} \psi(t)=\infty$. Let $\sigma=((p_0,q_0)^T, N)$ be as above. It follows that for almost all $x\in \RR$, as $T\rightarrow \infty$,
\[
\#\left(\left\{\left(\begin{array}{c} p \\ q \end{array}\right)\in P(\ZZ^2): \begin{array}{c}
\left|qx-p\right|<\psi(|q|), |q|\le T\\[0.05in]
(p,q)^T\equiv (p_0,q_0)^T\;(\mod N) \end{array}\right\}\right)\sim\frac 1 {\zeta_N(2)N^2}\sum_{1\le t\le T} \psi(t).
\]
\end{theorem}

\subsection*{Organization}
In Section 2, we review relevant previous works on counting problems and Diophantine approximation under congruence conditions.
In Section 3, we establish moment formulas (Theorem~\ref{thm: moment formulas}) for Siegel transforms with both primitive and congruence conditions. A central component of this section is to derive the second moment formula. To this end, we determine the set of possible determinants of $(\bf{v}_1, \bf{v}_2)\in P(\ZZ^2)\times P(\ZZ^2)$ (as $2\times 2$ matrices) with $\bf{v}_1\equiv\bf{v}_2\;(\mod N)$, and compute the number of $\Gamma(N)$-orbits in an $\SL_2(\RR)$-orbit at $(\bf{v}_1, \bf{v}_2)$ (see Proposition~\ref{pr:no.inv}). This computation plays a crucial role in establishing the second moment formula over the cone in the next section.
In Section 4, we introduce the cone $\mathfrak C$ of (a fundamental domain of) $\SL_2(\RR)/\Gamma(N)$, and derive moment formulas over the cone $\mathfrak C$ (Theorem~\ref{thm: moment formulas over the cone}). 
Finally in Section 5, we outline the proofs of the random counting theorem under primitive and congruence conditions (Theorem~\ref{thm: Analog of Schmidt}) and quantitative Khintchine theorem under congruence conditions (Theorem~\ref{thm: quantitative K}).
%

\subsection*{Acknowledgments}
We would like to thank the organizers of the conference ``Women in Dynamical Systems and Ergodic Theory'' held at Centro di Ricerca Matematica Ennio De Giorgi, which inspired the initiation of this project.
We are grateful for valuable discussions and helpful comments from Bence Borda, Claire Burrin, Samantha Fairchild, and Barak Weiss.
This work was supported by a New Faculty Research Grant of Pusan National University, 2025 for the first author.
The second author was supported by BK21 SNU Mathematical Sciences Division and National Research Foundation of Korea, under Project number RS-2025-00515082.

\section{Motivations and related results} 

\subsection*{Best approximation with a congruence condition}
The best approximations of $\alpha\in\mathbb R$ are the sequence of the rational numbers $p/q$ minimizing $|q\alpha-p|$ among the rationals with denominators at most $q$. 
It is known that the continued fraction gives all the best approximations for each real numbers.
In the real plane, since the distance between an integer vector $(p,q)^T$ and the line $y=\alpha^{-1} x$ is $(\alpha^2+1)^{-1/2}\cdot|q\alpha-p|$, each best approximation corresponds to the closest lattice point $(p,q)^T$ on the set $\{(x,y)^T\in \ZZ^2:0<y\le q\}$.

Best approximations subject to certain congruence conditions has been studied via appropriate continued fraction algorithms.
The even-integer continued fraction, an unfolded version of the even continued fraction introduced by Schweiger \cite{Sch82,Sch84}, generates all best rational approximations whose numerators and denominators have opposite parity, i.e., of the form even/odd or odd/even \cite{SW16}.
This condition corresponds to the congruence condition by $\sigma=((0,1)^T,2)$ or $((1,0)^T,2)$.

Kim, Liao and the second author introduced a continued fraction algorithm detecting the best approximations of the form odd/odd corresponding to the congruence condition $\sigma=((1,1)^T,2)$.
They further explored best approximations among rationals satisfying various congruence conditions modulo $2$ in \cite{KLL22,KLL25}.
All the congruence conditions modulo $2$ are represented by a proper nonempty subset of the set
$$\left\{\left(\begin{pmatrix}
0\\
1\end{pmatrix},2\right),\  \left(\begin{pmatrix}
1\\
0\end{pmatrix},2\right), \ \left(\begin{pmatrix}
1\\
1\end{pmatrix},2\right)\right\}.$$
In these works, they provided continued fraction algorithms that generate best approximations under certain parity constraints by using the boundary expansion of the triangle group in $\PGL_2(\ZZ)$ of the ideal triangle in the hyperbolic plane $\HH^2$, whose vertices are $0,1,\infty$.

In the other context, for the case of a general modulus $N$, the asymptotic frequency of the best approximations under a congruence condition is studied via the congruence subgroups. 
Two different approaches appear in \cite{Moe82} and \cite{JL88}.
The essential distinction between the two lies in the treatment of congruence conditions: the former imposes congruence modulo sign, considering rationals $p/q$ such that
\begin{equation}\label{congruence condition with sign}
\begin{pmatrix}
p\\ q
\end{pmatrix}\equiv \pm \bf{v}_0\;\;\;\; (\mod N)
\end{equation}
for a given $\bf{v_0}\in P(\ZZ^2)$ and $N\in \NN$, whereas the latter requires the exact congruence $(p,q)^T\equiv \bf{v}_0\;(\mod N)$, distinguishing $\bf{v}_0$ from its negative.
Improving upon their result, the central limit theorem and the law of the iterated logarithm are established in \cite{Bor25}.

Both results yield that the given congruence condition modulo $N$ is uniformly distributed, i.e., the proportion of best approximations satisfying the given congruence condition among all best approximations tends to the reciprocal of the number of possible congruence classes modulo $N$.
The congruence analogue of the quantitative Khintchine theorem (Theorem~\ref{thm: quantitative K}) involves the aforementioned asymptotic frequency.
The constant $1/(\zeta_N(2)N^2)$ in the theorem is the product of the primitive factor $1/\zeta(2)$ and the asymptotic frequency of best approximants under the congruence condition $(p,q)^T \equiv \mathbf{v}_0 \ (\mathrm{mod}\ N)$.
We also note that Fuchs proved central limit theorems for Khintchine theorem in the setting where only the denominator $q$ is subject to a congruence condition \cite{Fuchs1, Fuchs2, Fuchs3}.

\subsection*{Moments for Siegel transforms with primitive/congruence conditions}
The set $P(\ZZ^d)$ of primitive integer vectors is the image set $\SL_d(\ZZ).\bf{e}_1$, where $\bf{e}_1$ is the first element of the canonical basis of $\RR^d$. In general, one can define a Siegel transform for any lattice subgroup $\Gamma$ of any group $G$ when there is a $\Gamma$-invariant discrete set $\Lambda$ in the given $G$-space. 

In our case, one set $\Gamma=\Gamma(N)$, the principal congruence subgroup of level $N$, and take $\Gamma(N)$-invariant discrete sets $\Lambda=\{\bf{v}\in P(\ZZ^2): \bf{v}\equiv \bf{v}_0\;(\mod N)\}$ which consists of finite number of $\Gamma(N)$-orbits. 

In \cite{BFC}, Burrin and Fairchild established Siegel--Veech type formulas for Siegel transforms associated to lattice subgroups $\Gamma\le \SL_2(\RR)$ containing $-I$, and single $\Gamma$-orbits in $\RR^2$ which are discrete. These integral formulas were applied to derive asymptotic estimates for the number of saddle connections on certain classes of translation surfaces.
We note that when considering the congruence condition \eqref{congruence condition with sign}, the associated lattice subgroup is the conjugate of $\Gamma_1(N) \cup (-\Gamma_1(N))$, and hence the results in \cite{BFC} can be used to obtain analogs of the theorems in Section~\ref{Sect: Introduction}.

In \cite{FH}, Fairchild and the first author investigated the $S$-arithmetic primitive Siegel transform and established its first and second moment formulas for $d\ge 2$. These moment formulas were then used to derive an $S$-arithmetic primitive analog of Schmidt's random counting theorem. As corollaries, they obtained the higher dimensional cases ($d\ge 3$) of Theorem~\ref{thm: Analog of Schmidt} (\cite[Theorem 1.3]{FH}) and Theorem~\ref{thm: quantitative K} (\cite[Theorem 1.4]{FH}) stated in Introduction (in their notation, $\zeta_N(d)=\zeta_S(d)$ and $S=\{\infty\}\cup\{p: \text{prime s.t. }p|N\}$).

In the setting where only a congruence condition is imposed, Ghosh, Kelmer and Yu \cite{GKY2022} defined a Siegel transform associated with the given congruence condition and obtained the first moment formula for $d\ge 2$ and the second moment formula for $d\ge 3$. We also note that Alam, Ghosh and the first author computed higher moment formulas for this Siegel transform in \cite{AGH2024}.
For a random quantification of Khintchine--Groshev theorem with congruence conditions,  Alam, Ghosh and Yu \cite{AGY2021} obtained the result for $d\ge 3$.
Recently, Alam and St\"{o}mbergsson \cite{AS} obtained counting results over the complex field with congruence conditions by analyzing Siegel transforms and their moment formulas in the $S$-arithmetic setting over (purely imaginary) number fields.

\section{Moment Formulas on $\SL_2(\RR)/\Gamma(N)$}

For the given congruence condition $\sigma=(\bf{v}_0, N)$, where $\bf{v}_0\in \ZZ^2$ and $N\in \NN$ with $\gcd(\bf{v}_0, N)=1$, denote by
\[
P^{(\sigma)}=\{\bf{v}\in P(\ZZ^2): \bf{v}\equiv \bf{v}_0\;(\mod N)\}
\]
which is invariant under the linear action of the principal congruence subgroup $\Gamma(N)$ of level $N$.
Our Siegel transform $\widehat{f}^{(\sigma)}$ is then defined as
\[
\widehat{f}^{(\sigma)}(g\Gamma(N))= \sum_{\bf{v}\in P^{(\sigma)}} f(g\bf{v}),\quad \forall g\Gamma(N)\in \SL_2(\RR)/\Gamma(N),
\]
for a bounded and compactly supported function $f:\RR^2\rightarrow \RR$. Since $\widehat{f}^{(\sigma)}$ is bounded by the primitive Siegel transform $\widehat{f}$, which is well known to be a bounded function (on $\SL_2(\RR)/\SL_2(\ZZ)$, though it can be extended to a function on $\SL_2(\RR)/\Gamma(N)$), one can conclude that $\widehat{f}^{(\sigma)}$ is also bounded hence lies in $L^k(\SL_2(\RR)/\Gamma(N))$ for any $k\ge 1$. 
Therefore, the $k$-th moment of $\widehat{f}^\sigma$ exists for every $k\in \NN$, and we will examine the integral formulas for the first and second moments.

\subsection{First Moment Formulas}
For $f:\RR^2\rightarrow \RR$, bounded and compactly supported, we shall show that
\[
\int_{\SL_2(\RR)/\Gamma(N)} \Siegel{f}{\sigma}(g\Gamma(N)) d\nu_N(g)
=\frac 1 {\zeta_N(2) N^2} \int_{\RR^2} f({\bf v}) d{\bf v},
\]
where $\nu_N$ is the Haar probability measure on $\SL_2(\RR)/\Gamma(N)$.

Let us begin with the following lemma, which is a standard result in number theory. 

\begin{lemma} \label{le:no.PN}
Let $\mathcal P_N$ be the set of representatives for the congruence classes modulo N for which $P^{(\sigma)}\neq \emptyset$.
One can take $\mathcal P_N$ as
$$\mathcal P_N=\{\sigma=((m,n),N):0\le m,n\le N-1 \text{ s.t. }\mathrm{gcd}(m,n,N)=1\}.$$

In particular, the number of elements in $\mathcal P_N$ is
\[
\# \mathcal P_N=N^2\prod_{p|N} \left(1-\frac 1 {p^2}\right).
\]
\end{lemma}

%

Let $G$ be a Lie group and let $\Gamma$ be a discrete subgroup of $G$.
If $\mathfrak{F}$ is a fundamental domain of $G/\Gamma$, then for any $g\in G$, the left translate $g\mathfrak{F}$ is also a fundamental domain.
In contrast, the right translate $\mathfrak{F}g$ is not, in general, a fundamental domain of $G/\Gamma$. However, if $g\in N_G(\Gamma)$, the normalizer of $\Gamma$ in $G$, then $\mathfrak {F}g$ is again a fundamental domain for $G/\Gamma$.

\begin{proposition} \label{pr:fund}
Let $G$ be a Lie group and let $\Gamma$ be a discrete subgroup of $G$. Denote by $\mathfrak F$ a fundamental domain for $G/\Gamma$. 
For any element $g\in N_G(\Gamma)$, the right translate $\mathfrak Fg$ is again a fundamental domain for $G/\Gamma$. 
\end{proposition}

\begin{proof}
Since $G =\bigsqcup_{\gamma\in \Gamma}\mathfrak{F}\gamma$ and $g^{-1}\Gamma g=\Gamma$, it follows that
$$G = Gg = \bigsqcup_{\gamma\in \Gamma}\mathfrak{F}\gamma g  
= \bigsqcup_{\gamma\in \Gamma}\mathfrak{F}g(g^{-1}\gamma g)
= \bigsqcup_{\gamma'\in \Gamma}\mathfrak{F}g \gamma'.$$
Thus $\mathfrak{F}g$ is a fundamental domain for $G/\Gamma$.
\end{proof}

\begin{proof}[Proof of Theorem~\ref{thm: moment formulas} (the 1st moment)]
For a bounded function $f:\RR^2\rightarrow \RR$ of compact support, we aim to show that
\[
\int_{\SL_2(\RR)/\Gamma(N)} \Siegel{f}{\sigma}(g\Gamma(N)) d\nu_N(g)
=\frac 1 {\zeta(2) N^2\prod_{p|N} \left(1-\frac 1 {p^2}\right)} \int_{\RR^2} f({\bf v}) d{\bf v}
=\frac 1 {\zeta_N(2) N^2} \int_{\RR^2} f({\bf v}) d{\bf v}.
\]

By Riesz-Markov-Kakutani representation theorem (which applies to $f\in C_c(\RR^2)$, and can be extended to bounded and compactly supported functions), there exists $\omega_\sigma>0$ for each $P^{(\sigma)}\in \mathcal P_N$ such that
$$\int_{\mathrm{SL}_2(\mathbb R)/\Gamma(N)}\widehat{f}^{(\sigma)}(g\Gamma(N))d\nu_N(g)= \omega_\sigma\int_{\mathbb R^2}f(\mathbf v)d\mathbf v.$$

Observe that $\widehat{f}$ and $\widehat{f}^{(\sigma)}$ for any $\sigma\in\mathcal P_N$ are bounded thus absolutely integrable, and
$\widehat{f}(g\Gamma(N)) = \sum_{\sigma\in\mathcal P_N}\widehat{f}^{(\sigma)}(g\Gamma(N))$ for $g\in\mathrm{SL}_2(\mathbb R)$. By Fubini lemma, it follows from the first moment formula for the primitive Siegel transform (see Eq.~(20) in \cite{Siegel45}) that
\begin{align*}
\frac{1}{\zeta(2)}\int_{\mathbb R^2}f(\mathbf v)d\mathbf v 
& = \int_{\mathrm{SL}_2(\mathbb R)/\mathrm{SL}_2(\mathbb Z)}\widehat{f}(g\SL_2(\mathbb Z))d\mu(g) = \int_{\mathrm{SL}_2(\mathbb R)/\Gamma(N)}\widehat{f}(g\Gamma(N))d\nu_N(g) \\
& = \sum_{ \sigma\in\mathcal{P}_N}\int_{\mathrm{SL}_2(\mathbb R)/\Gamma(N)}\widehat{f}^{(\sigma)}(g\Gamma(N))d\nu_N(g) = \sum_{\sigma\in\mathcal{P}_N}\omega_\sigma\int_{\mathbb R^2}f(\mathbf v)d\mathbf v.
\end{align*}
Here, $\mu$ is the Haar probability measure on $\mathrm{SL}_2(\mathbb R)/\mathrm{SL}_2(\mathbb Z)$.
Thus we obtain the equality
$$\sum_{\sigma\in\mathcal{P}_N} \omega_\sigma = \frac{1}{\zeta(2)}.$$

Now we claim that $\omega_\sigma=\omega_\tau$ for all $\sigma,\;\tau\in\mathcal{P}_N$, so that the theorem follows from Lemma~\ref{le:no.PN}.

Choose $(m_1,n_1)\in P^{(\tau)}$ and $(m_2,n_2)\in P^{(\sigma)}$.
There exist $a,b,c,d\in\mathbb Z$ such that $am_1+bn_1=1$ and $cm_2+dn_2=1$.
If we take $g_{\sigma\tau}=\begin{pmatrix}m_2&-d\\n_2&c\end{pmatrix}\begin{pmatrix}a&b\\-n_1&m_1\end{pmatrix}$,
then
$$g_{\sigma\tau}\begin{pmatrix}m_1\\n_1\end{pmatrix} = \begin{pmatrix}m_2&-d\\n_2&c\end{pmatrix}\begin{pmatrix}1\\0\end{pmatrix}=\begin{pmatrix}m_2\\n_2\end{pmatrix} \text{ and }g_{\sigma\tau}\in\mathrm{SL}_2(\mathbb R).$$
Since $g_{\sigma\tau}$ preserves the set $N\ZZ^2$ and sends primitive vectors to primitive vectors, we obtain that $g_{\sigma\tau}P^{(\tau)}=P^{(\sigma)}$.

Let $\mathfrak{F}$ be a fundamental domain of $\mathrm{SL}_2(\mathbb R)/\Gamma(N)$.
By unimodularity of $\mathrm{SL}_2(\mathbb R)$ and Lemma~\ref{pr:fund}, it holds that for any bounded and compactly supported function $f$ on $\RR^2$,
\begin{align*}
\omega_\sigma\int_{\mathbb R^2}f(\mathbf v)d\mathbf v 
& = \int_{\mathrm{SL}_2(\mathbb R)/\Gamma(N)}\widehat{f}^{(\sigma)}(g\Gamma(N))d\nu_N(g) = \int_{\mathfrak{F}}\sum_{\mathbf v\in\mathcal{P}^{(\sigma)}}f(g\mathbf v)d\nu_N(g) \\
& = \int_{\mathfrak{F}}\sum_{\mathbf v\in\mathcal{P}^{(\tau)}}f(gg_{\sigma\tau}\mathbf v)d\nu_N(g) = \int_{\mathfrak{F}g_{\sigma\tau}}\sum_{\mathbf v\in\mathcal{P}^{(\tau)}}f(g'\mathbf v)d\nu_N(g') \\
& = \int_{\mathrm{SL}_2(\mathbb R)/\Gamma(N)}\widehat{f}^{(\tau)}(g\Gamma(N))d\nu_N(g) = \omega_\tau\int_{\mathbb R^2}f(\mathbf v)d\mathbf v.
\end{align*}
Thus $\omega_\sigma = \omega_\tau$ for any $\sigma, \tau \in \mathcal P_N$. From the definition of zeta functions $\zeta(d)$ and $\zeta_N(d)$, it follows that
\[
\zeta(2)N^2\prod_{p|N} \left( 1 - \frac 1 {p^2} \right)= \zeta_N(2)N^2.
\]
\end{proof}

\subsection{Second Moment Formulas}

Let $F$ be a bounded and compactly supported function on $\RR^2\times \RR^2$. For $N\in \NN$ and $\sigma\in \mathcal P_N$, define $\widehat{F}^{(\sigma)}$, a function on $\SL_2(\RR)/\Gamma(N)$ by
\[
\widehat{F}^{(\sigma)}(g\Gamma(N))=\sum_{\bf{v}_1, \bf{v}_2\in P^{(\sigma)}} F(g\bf{v}_1, g\bf{v}_2).
\]

Note that if we take $F=f\otimes f$, i.e., $F(\bf{v}_1, \bf{v}_2)=f(\bf{v}_1)f(\bf{v}_2)$, then
\[
\widehat{F}^{(\sigma)}(g\Gamma(N))=\widehat{f}^{(\sigma)}(g\Gamma(N))^2.
\]
Hence we obtain the absolute integrability of $\widehat{F}^{(\sigma)}$ by choosing an appropriate bounded and compactly supported function $f:\RR^2\rightarrow \RR$ satisfying that $|F|\le f\otimes f$.

For the next theorem, let us introduce the Haar measure $\eta$ on $\SL_2(\RR)$, inherited from $\RR^2\times \RR^2$, defined as $dadbdc$ if we denote a (generic) element $g$ of $\SL_2(\RR)$ by
\[
g=\left(\begin{array}{cc}
1 & 0 \\
c & 1\end{array}\right)\left(\begin{array}{cc}
a & b \\
0 & a^{-1}\end{array}\right).
\]
It is easy to verify that
\[
\nu_N=\frac 1 {[\SL_2(\ZZ):\Gamma(N)]} \mu= \frac 1 {N^3 \prod_{p|N} \left(1 -\frac 1 {p^2}\right)}\cdot \frac 1 {\zeta(2)} \eta=\frac 1 {\zeta_N(2)N^3} \eta.
\]

\begin{theorem}\label{thm: second moment 2}
Let $N\in \NN_{\ge 2}$ and fix a congruence class $\sigma\in \mathcal P_N$. For a bounded and compactly supported function $F:\RR^2\times \RR^2\rightarrow \RR$, it follows that
\[\begin{split}
\int_{\SL_2(\RR)/\Gamma(N)} \widehat{F}^{(\sigma)} (g\Gamma(N)) d\nu_N (g)
&=\sum_{n\in N\ZZ-\{0\}} \frac {\varphi(n)} {\zeta_N(2)N^3\varphi(N)}\int_{\SL_2(\RR)} F(gJ_n) d\eta(g) \\
&+\frac 1 {\zeta_N(2) N^2} \int_{\RR^2} F(\bf{v},\bf{v})+F(\bf{v},-\bf{v}) d\bf{v},
\end{split}\]
where for each $n\in N\ZZ-\{0\}$, $J_n=\left(\begin{array}{cc}
1 & 0 \\
0 & n\end{array}\right)$.
\end{theorem}

\begin{proof}[Proof of Theorem~\ref{thm: moment formulas} (the 2nd moment).] This is the special case of Theorem~\ref{thm: second moment 2}, where we take $F=f\otimes f$.
\end{proof}

For any $n\in \ZZ-\{0\}$ and $\sigma_i,\; \sigma\in \mathcal P_N$, define
\begin{equation}\label{def: admissible det}
D^{(\sigma_1, \sigma_2)}_n=\left\{({\bf{v_1}}, {\bf{v_2}})\in P^{(\sigma_1)}\times P^{(\sigma_2)}: 
\det({\bf{v_1}}, {\bf{v_2}})=n \right\}
\end{equation}
and $D^{(\sigma)}_n=D^{(\sigma, \sigma)}_n$.

\begin{proposition}\label{pr:no.inv}
The set $D_n^{(\sigma)}\neq \emptyset$ only when $n\in N\ZZ-\{0\}$. The number of $\Gamma(N)$-invariant irreducible sets in $D_n^{(\sigma)}$ for $n\in N\ZZ-\{0\}$ is $N \varphi(n)/\varphi(N)$, where $\varphi(n):=\varphi(|n|)$ is the Euler totient function. 

\end{proposition}

It is worth noting that when $N=2$, stronger results are available:
\[
D_n^{(\sigma_1, \sigma_2)}\neq \emptyset\; \text{if and only if}\; \left\{\begin{array}{ll}
\sigma_1=\sigma_2  &\text{when}\;n\in 2\ZZ-\{0\}; \\
\sigma_1\neq \sigma_2  &\text{when}\;n\in 2\ZZ+1. \\
\end{array}\right.
\]

For any $\sigma\in \mathcal P_N$, if $(\bf{v}_1, \bf{v}_2)\in D_n^{(\sigma)}$, then $\bf{v}_1\equiv \bf{v}_2$ modulo $N$ so that $n=\det(\bf{v}_1, \bf{v}_2)\in N\ZZ$, which proves the first assertion of Proposition~\ref{pr:no.inv}. For the second assertion, we need the following lemma.

\begin{lemma}[Chinese remainder theorem (non-coprime version)]\label{le:CRT}
For $a,b,m,n\in\mathbb Z$ such that $a\equiv b \pmod{\mathrm{gcd}(m,n)}$, the system
$$\begin{cases}x\equiv a\pmod{m},\\
x\equiv b\pmod{n}\end{cases}$$
has a solution, which is unique up to being modulo $\mathrm{lcm}(m,n)$.
\end{lemma}

\begin{proof}[Proof of Proposition~\ref{pr:no.inv}]
We already showed the first assertion of the proposition.
Fix  $n\in N\mathbb Z-\{0\}$. It follows from \cite{Schmidt1960} that
$$\left\{(\mathbf v_1,\mathbf v_2)\in P(\mathbb Z^2)\times P(\mathbb Z^2):\mathrm{det}(\mathbf v_1,\mathbf v_2)=n\right\}=\bigsqcup_{\substack{1\le\ell\le n\\ \mathrm{gcd}(\ell,n)=1}}\mathrm{SL}_2(\mathbb Z).\begin{pmatrix}1&\ell\\0&n\end{pmatrix}.$$
Denote by \{$\gamma_i : 1\le i\le [\SL_2(\ZZ):\Gamma(N)]\}$ the set of representatives of $\Gamma(N)\backslash\mathrm{SL}_2(\mathbb Z)$, i.e.
$$\mathrm{SL}_2(\mathbb Z) = \bigsqcup_{i=1}^{[\SL_2(\ZZ):\Gamma(N)]}\Gamma(N).\gamma_i.$$

Hence for each $\sigma\in \mathcal P_N$,
\[
D^{(\sigma)}_n=\left\{(\mathbf v_1,\mathbf v_2)\in P^{(\sigma)}\times P^{(\sigma)}:\mathrm{det}(\mathbf v_1,\mathbf v_2)=n\right\}=\bigsqcup \Gamma(N).\gamma_i\begin{pmatrix}1&\ell\\0&n\end{pmatrix},
\]
where the (disjoint) union is taken over those $\gamma_i\begin{pmatrix}1&\ell\\0&n\end{pmatrix}$ whose first and second columns are contained in the same congruence class $\sigma$ and with $\gcd(\ell, n)=1$.
For each $\sigma\in \mathcal P_N$, the number of $\gamma_i$ whose first column vector lies in $P^{(\sigma)}$ is
\begin{equation*}\label{eqn 1:no.inv}
[\SL_2(\ZZ):\Gamma(N)]/\#\mathcal P_N=N.
\end{equation*} 

Moreover, if we let $\gamma_i = \begin{pmatrix}a&c\\b&d\end{pmatrix},$ then 
\[
\gamma_i\begin{pmatrix}1&\ell\\0&n\end{pmatrix}=\begin{pmatrix}a&a\ell+cn\\b&b\ell+dn\end{pmatrix}\equiv\begin{pmatrix}a&a\ell\\b&b\ell\end{pmatrix}\pmod{N},
\]
thus the above matrix is in $D^{(\sigma)}_n$ if and only if its first and second columns are contained in the same congruence class $\sigma$.



Denote $n=Nm$ for some $m\in \ZZ-\{0\}$. We claim that
\begin{equation*}
\begin{pmatrix}a\ell\\b\ell\end{pmatrix}\equiv \begin{pmatrix}a\\ b\end{pmatrix}\pmod{N}\quad\text{ if and only if }\quad\ell= kN+1\text{ for some }1\le k\le m-1,
\end{equation*}
where the reverse direction is trivial.

Suppose that $a(\ell-1)\equiv b(\ell-1)\equiv 0\pmod{N}$.
If $N\nmid\ell-1$, there exists a divisor $p\not=1$ of $N$ such that $p$ divides both $a$ and $b$, which contradicts to $\mathrm{gcd}(a,b)=1$.
Thus we have $N|\ell-1$. Since $1\le\ell\le Nm$, we have $0\le k \le m-1$.

Set
\begin{equation*}\label{eq:SNm1}
S_N(m):=\{1\le \ell\le Nm: \ell\equiv 1\ (\mathrm{mod} \ {N}),~\mathrm{gcd}(\ell,Nm)=1\}
\end{equation*}
(the second condition comes from the fact that $(\ell, n)$ is a primitive vector).
Hence the number of $\Gamma(N)$-invariant irreducible sets in $D^{(\sigma)}_n$ is 
\begin{equation*}\label{eqn 2:no.inv}
\frac{[\mathrm{SL}_2(\mathbb Z):\Gamma(N)]}{\# \mathcal P_N}\cdot|S_N(m)|=N\cdot|S_N(m)|.
\end{equation*}

Now, let us show that
\begin{equation}\label{eq:no.inv}
|S_N(m)|=\frac{\varphi(Nm)}{\varphi(N)}=\frac{\varphi(n)}{\varphi(N)}.
\end{equation}

Under the assumption that $\ell\equiv 1\pmod{N}$, the condition
$\mathrm{gcd}(\ell,Nm)=1$ is equivalent to the condition $\mathrm{gcd}(\ell,m)=1$
since $\mathrm{gcd}(\ell,N)=1$.
Thus
$$S_N(m) = \{1\le \ell \le Nm: \ell\equiv 1\ (\mathrm{mod} \ {N}),~\mathrm{gcd}(\ell,m)=1\}.$$

We will achieve \eqref{eq:no.inv} by induction on $d := \mathrm{gcd}(N,m)$, starting with two base cases: i) $d=1$ and ii) $d=m$.

\vspace{0.1in}
\noindent i)
Suppose that $d=1$.

If $m=1$, then $S_N(1) = \{1\}$ and automatically \eqref{eq:no.inv} holds.

Suppose $m\not=1$. Consider the projection map
\[
\ZZ/n\ZZ \rightarrow \ZZ/N\ZZ \times \ZZ/m\ZZ:\;\ell \mapsto (\ell \;\mod N, \ell \;\mod m)=(x_\ell,y_\ell).
\]
Then if $\ell\in S_N(m)$, then $x_\ell=1$ and $\gcd(y_\ell, m)=1$. 
On the other hand, it follows from the Chinese remainder theorem that for any integer $y$, the system
\begin{equation*}\begin{cases}
\ell\equiv 1 \pmod N,\\
\ell\equiv y \pmod m
\end{cases}\end{equation*}
has a solution $\ell$ which is unique up to being modulo $n=Nm$.
Thus
$$|S_N(m)|=\#(\mathbb Z/m\mathbb Z)^\times=\varphi(m).$$
By our assumption that $d=\gcd(N,m)=1$, we conclude that
$$\frac{\varphi(Nm)}{\varphi(N)} = \frac{\varphi(N)\varphi(m)}{\varphi(N)}=\varphi(m)=|S_N(m)|.$$

\vspace{0.1in}
\noindent ii) $d=m$, i.e., $m|N$. 

Note that $\ell\equiv 1 \pmod{N}$ implies $\mathrm{gcd}(\ell,N)=1$.
Thus we have $\mathrm{gcd}(\ell,Nm)=1$ and hence $|S_N(m)|=m$.
It follows that
$$\frac{\varphi(Nm)}{\varphi(N)}= \frac {\varphi(N)\varphi(m)\gcd(N,m)}{\varphi(N)\varphi(\gcd(N,m))}= \frac{\varphi(N)\varphi(m)m}{\varphi(N)\varphi(m)} = m=|S_N(m)|.$$

\vspace{0.1in}
\noindent iii) (induction step) Now, let us assume that $\mathrm{gcd}(N,m)=d$ with $1<d<m$.

Set $m'=m/d$. Define the modulo $m'$ map 
\begin{equation*}
\begin{matrix}
\psi:& S_N(m) & \to & S_d(m')=\{1\le \ell\le dm': \ell\equiv 1 \;(\mathrm{mod}\; d),\;\gcd(\ell, m')=1\} \\[0.05in]
& \ell & \mapsto & \hspace{-2.9in}\ell \mathrm{~mod~} m.
\end{matrix}
\end{equation*}
To verify that $\psi$ is well-defined, let $\ell=jN+1\in S_N(m)$ for some $0\le j\le m-1$.
Since $d|N$, it is obvious that $\ell\equiv 1$ modulo $d$. And from $\gcd(\ell, m)=1$ together with $\gcd(\ell, d)=1$, we obtain the second condition $\gcd(\ell, m')=1$.


We claim that $\psi$ is a $d$-to-1 surjective map. Indeed, Chinese remainder theorem (Lemma~\ref{le:CRT}) tells us that for  any $u \in S_d(m')$, the following system
\begin{equation*}
\begin{cases}
\ell\equiv u\pmod m,\\
\ell\equiv 1\pmod N
\end{cases}
\end{equation*}
has a solution, unique up to being modulo $\mathrm{lcm}(m,N)={mN}/{d}$.
Apparently, such a solution $\ell\in S_N(m)$ satisfies that $\psi(\ell)=u$ which shows the surjectivity of $\psi$.

Notice that we reduce the problem counting $|S_N(m)|$ to counting $|S_d(m')|$, where $1\le m'<m$. Repeating this process, we eventually end up in finite steps to the base case, either i) or ii). 
Therefore, we can apply an induction hypothesis which asserts that 
$$|S_d(m')|=\frac{\varphi (dm')}{\varphi(d)}=\frac{\varphi(m)}{\varphi(d)},$$
thus
$|S_N(m)|={d\varphi(m)}/{\varphi(d)}.$
Consequently, we achieve the claim \eqref{eq:no.inv} by observing
$$\frac{\varphi(n)}{\varphi(N)} = \frac{\varphi(Nm)}{\varphi(N)}=\frac{\varphi(N)\varphi(m)d}{\varphi(N)\varphi(d)} = \frac{\varphi(m)d}{\varphi(d)}=|S_N(m)|.$$
\end{proof}


We now ready to prove Theorem~\ref{thm: second moment 2}.
\begin{proof}[Proof of Theorem~\ref{thm: second moment 2}]
By \Cref{pr:no.inv}, for each $\sigma\in \mathcal P_N$, one can decompose $P^{(\sigma)}\times P^{(\sigma)}$ into disjoint $\Gamma(N)$-invariant sets
\[
P^{(\sigma)}\times P^{(\sigma)}=\left\{({\bf v}, {\bf v}): {\bf v}\in P^{(\sigma)}\right\}\sqcup \left\{({\bf v}, -{\bf v}): {\bf v}\in P^{(\sigma)}\right\}\sqcup \bigsqcup_{n\in N\ZZ-\{0\}} D^{(\sigma)}_n,
\]
where $D^{(\sigma)}_n$ is defined as in \eqref{def: admissible det}. One can apply Fubini lemma so that
\begin{align}
&\int_{\SL_2(\RR)/\Gamma(N)} \Siegel{F}{\sigma}(g\Gamma(N)) d\nu_N (g)
=\int_{\SL_2(\RR)/\Gamma(N)} \sum_{
\bf v_1, \bf v_2\in P^{(\sigma)}
} F(g{\bf v}_1, g{\bf v}_2) d\nu_N(g) \nonumber\\
&\hspace{0.4in}=\int_{\SL_2(\RR)/\Gamma(N)} \sum_{{\bf v}\in P^{(\sigma)}} F(g{\bf v}, g{\bf v}) d\nu_N(g)+\int_{\SL_2(\RR)/\Gamma(N)} \sum_{{\bf v}\in P^{(\sigma)}} F(g{\bf v}, -g{\bf v}) d\nu_N(g) \label{dependent}
\\
&\hspace{0.6in}+\sum_{n\in N\ZZ-\{0\}}\int_{\SL_2(\RR)/\Gamma(N)} \sum_{
({\bf v}_1, {\bf v}_2)\in
D^{(\sigma)}_n} F(g{\bf v}_1, g{\bf v}_2) d\nu_N(g). \label{independent}
\end{align}

Applying the first moment formula (Theorem~\ref{thm: moment formulas} (1)) to bounded and compactly supported functions
\begin{equation}\label{eq:F.fun}
{\bf x} \mapsto F({\bf x}, {\bf x})
\quad\text{and}\quad
{\bf x} \mapsto F({\bf x}. -{\bf x}) 
\end{equation}
on $\RR^2$, it follows that
\[
\eqref{dependent}=\frac 1 {\zeta_N(2)N^2} \int_{\RR^2} F({\bf v}, {\bf v}) d{\bf v}
+\frac 1 {\zeta_N(2)N^2} \int_{\RR^2} F({\bf v}, -{\bf v}) d{\bf v}.
\] 

For integrals in \eqref{independent}, it follows from \Cref{pr:no.inv} that for each $n\in N\ZZ-\{0\}$, $D^{(\sigma)}_n$ consists of $N\varphi(n)/\varphi(N)$ number of $\Gamma(N)$-orbits. Denote by $D^{(\sigma)}_n=\bigsqcup_{i=1}^{N\varphi(n)/\varphi(N)} \Gamma(N) X_i$, where each $X_i\in \ZZ^2\times \ZZ^2$ with $\det X_i=n$.

Applying the folding-unfolding argument,
\[\begin{split}
\int_{\SL_2(\RR)/\Gamma(N)} \sum_{
({\bf v}_1, {\bf v}_2)\in
D^{(\sigma)}_n} F(g{\bf v}_1, g{\bf v}_2) d\nu_N(g)
&=\sum_{i=1}^{N\frac{\varphi(n)}{\varphi(N)}} \int_{\SL_2(\RR)/\Gamma(N)} \sum_{
({\bf v}_1, {\bf v}_2)\in
\Gamma(N)X_i} F(g{\bf v}_1, g{\bf v}_2) d\nu_N(g)\\
&=\sum_{i=1}^{N\frac{\varphi(n)}{\varphi(N)}} \int_{\SL_2(\RR)} F(gX_i) d\nu_N(g).
\end{split}\]

Thus the result follows from the change of variables $gX_i\mapsto gJ_n$ ($\nu_N$: unimodular) and the relation $\nu_N=1/(\zeta(2)N^3\prod_{p|N} (1-\frac 1 {p^2}))\eta=1/(\zeta_N(2)N^3)\eta$.
\end{proof}

\section{Moment Formulas over the Cone}
In this section, we fix a fundamental domain $\mathfrak F_N$ of $\SL_2(\RR)/\Gamma(N)$. Note that $\mathfrak F_N$ can be considered as the subset of $\RR^2\times \RR^2$. 
We define the cone $\mathfrak C=\mathfrak C_{\mathfrak F_N}$ of $\SL_2(\RR)/\Gamma$ by the embedded image of the map
\[
(t, g) \in (0,1]\times \mathfrak F_N \mapsto t^{1/2}g\in \RR^2\times \RR^2
\]
and assign the product measure $\mathrm{Leb}|_{(0,1]}\times \nu_N$.

Note that the Siegel transform $\Siegel{f}{\sigma}$ and $\Siegel{F}{\sigma}$ can be extended to the cone $\mathfrak C$. However, in this setting, it is more convenient to work with a normalized version of the Siegel transform, defined as follows.
\[\begin{gathered}
t\Siegel{f}{\sigma}(t^{1/2}g):=t\sum_{{\bf v}\in P^{(\sigma)}} f(t^{1/2}g{\bf v})\\
t^2\Siegel{F}{\sigma}(t^{1/2}g):=t^2\hspace{-0.15in}\sum_{\substack{
({\bf v}_1, {\bf v}_2)\\
\in P^{(\sigma)}\times P^{(\sigma)}}} 
\hspace{-0.15in}
F(t^{1/2}g{\bf v}_1, t^{1/2}g{\bf v}_2)
\end{gathered}\]
for $(t,g)\in (0,1]\times \mathfrak F_N$.

\begin{theorem}\label{thm: First moment on C}
Let $f:\RR^2 \rightarrow \RR$ be a bounded and compactly supported function.
For $P^{(\sigma)}\in \mathcal P_N$, the function $t^{1/2}g\in \mathfrak C\mapsto t\widehat{f}^\sigma(t^{1/2}g)$ is integrable and we have the integral formula
\[
\int_{\mathfrak C} t\Siegel{f}{\sigma} (t^{1/2}g) d\nu_N dt
=\frac 1 {\zeta_N(2)N^2} \int_{\RR^2} f({\bf v}) d{\bf v}.
\]
\end{theorem}
\begin{proof} 
We will assume that $f$ is a non-negative function. The formula for a general function $f$ can be easily obtained since one can decompose $f=f_{+}-f_{-}$ by two non-negative, bounded, and compactly supported functions $f_+:=\max\{f, 0\}$ and $f_-:=\max\{-f, 0\}$.
Applying Tonelli's theorem,
\[
\int_{\mathfrak C} t\Siegel{f}{\sigma} (t^{1/2}g) d\nu_N dt=\int_0^1 t\int_{\SL_2(\RR)/\Gamma(N)}
\widehat{f_t}^{(\sigma)}(g)\nu_Ndt,
\]
where $f_t(\bf{v}):=f(t^{1/2}\bf{v})$.
The formula as well as the integrability follows directly from applying \Cref{thm: moment formulas} (1) in the inner integral, with the fact that
\begin{equation}\label{eq:ch.v}
\int_{\mathbb R^2} f(t^{1/2}\mathbf v)d\mathbf v = \int_{\mathbb{R}^2}t^{-1}f(\mathbf w)d\mathbf w.
\end{equation}
\end{proof}

\begin{theorem}\label{thm: Second moment on C}
For $N\in \NN_{\ge 2}$, define a function $\Phi_N$ on $\RR$ as
\begin{equation}\label{def: Phi_N}
\Phi_N(x)=\:|x|\sum_{\substack{
n\in N\NN,\\
n\ge |x|}} \frac {N\varphi(n)} {\varphi(N)n^3}.
\end{equation}
Let $F:\RR^2\times \RR^2 \rightarrow \RR$ be a bounded and compactly supported function.
For $P^{(\sigma)}\in \mathcal P_N$, we have the integral formula
\begin{equation}\label{eq 1: Cone second moment}\begin{split}
\int_{\mathfrak C} t^2\Siegel{F}{\sigma} (t^{1/2}g) d\nu_N dt
&=\frac 1 {\zeta_N(2)N^3} \int_{\RR^2} \Phi_N(\det({\bf v}_1, {\bf v}_2)) F({\bf v}_1, {\bf v}_2) d{\bf v}_1 d{\bf v}_2\\
&+\frac 1 {2\zeta_N(2)N^2} \int_{\RR^2} F({\bf v}, {\bf v}) + F({\bf v}, -{\bf v}) d{\bf v}.
\end{split}\end{equation}
\end{theorem}

\begin{proof}
As in the proof of \Cref{thm: First moment on C}, we will assume that $F$ is a non-negative function. 
By Tonelli's theorem, and applying \Cref{thm: moment formulas} (2) to $F_t(\bf{v}_1,\bf{v}_2):=F(t^{1/2}\bf{v}_1, t^{1/2}\bf{v}_2)$,
it holds that
\[\begin{split}
\int_{\mathfrak C} t^2 \Siegel{F}{\sigma} (t^{1/2}g)d\nu_N dt
&=\int_0^1 t^2 \left[
\sum_{n\in N\ZZ-\{0\}} \frac {\varphi(n)}{\zeta_N(2)N^2\varphi(N)}\int_{\SL_2(\RR)} F_t(gJ_n) d\eta(g)\right.\\
&\left. \hspace{1in}+\frac 1 {\zeta_N(2)N^2} \int_{\RR^2} F_t(\bf{v}, \bf{v}) + F_t(\bf{v}, \bf{-v})d\bf{v}\right]dt.
\end{split}\]
By applying \eqref{eq:ch.v} to the functions in \eqref{eq:F.fun},
the integral in the second line is equal to the integral in the second line above in \eqref{eq 1: Cone second moment}.
Hence let us concentrate on the integral in the first line of (R.H.S).

For $t\in (0,1]$ and $n\in N\ZZ-\{0\}$, consider the change of variables $g'=gh^{-1}_t$, where $h_t=\diag(t^{-1/2}, t^{1/2})$ so that 
\[
t^{1/2}gJ_n=g'\left(\begin{array}{cc}
1 & 0\\
0 & tn \end{array}\right).
\] 
Then putting $x=tn$, we deduce that the above first integral in (R.H.S) is equal to
\begin{equation}\label{eqn: 2nd moment first term}
\frac 1 {\zeta_N(2) N^2\varphi(N)} \sum_{n\in N\ZZ-\{0\}} \frac {\varphi(n)} { |n|^3}
\int_{n(0,1]} x^2 \int_{\SL_2(\RR)} F\left(g\left(\begin{array}{cc}
1 & 0 \\
0 & x\end{array}\right)\right) d\eta(g)dx.
\end{equation}
Here, $n(0,1]=(0,n]$ if $n>0$, and $[n,0)$ if $n<0$.

On the domain $n(0,1]\times \SL_2(\RR)$, we change variables
\[
g\left(\begin{array}{cc}
1 & 0 \\
0 & x \end{array}\right)=\left(\begin{array}{cc}
1 & 0 \\
c & 1 \end{array}\right)\left(\begin{array}{cc}
a & b \\
0 & a^{-1} \end{array}\right)\left(\begin{array}{cc}
1 & 0 \\
0 & x \end{array}\right)=(\bf{v}_1, \bf{v}_2)
\]
so that $$d\eta(g)dx=dadbdcdx=\frac 1 {|\det(\bf{v}_1, \bf{v}_2)|}d\bf{v}_1d\bf{v_2}.$$

Finally, let us switch two integrals.
Notice that
\[
(\bf{v}_1, \bf{v}_2)\in n(0,1]\times \SL_2(\RR)
\quad\text{if and only if}\quad
\left\{\begin{array}{cl}
x=\det(\bf{v}_1, \bf{v}_2) \ge n &\text{if}\; x>0\;(\text{hence}\; n>0);\\[0.05in]
x=\det(\bf{v}_1, \bf{v}_2) \le n &\text{if}\; x<0,
\end{array}\right.
\]
which explains the subscript of the summation in $\Phi_N(x)$.
Therefore the integral \eqref{eqn: 2nd moment first term} is reduced to
\[
\frac 1 {\zeta_N(2)N^3} \int_{\RR^2\times \RR^2} \Phi_N(\det(\bf{v}_1, \bf{v}_2)) F(\bf{v_1}, \bf{v_2}) d\bf{v}_1 d\bf{v}_2
\]
and we obtain the result.
\end{proof}


For the purpose of counting lattice points, it is necessary to obtain an estimate for $\Phi_N(x)$.

\begin{proposition}\label{Prop: asym formula of Phi_N} Let $\Phi_N(x)$ be defined as in \eqref{def: Phi_N} for each $N\in \NN$. It holds that
\[\Phi_N(x)=\dfrac 1 {\zeta_N(2)N} + O_N(|x|^{-1}\log |x|)\]
(as $|x|\rightarrow \infty$).
\end{proposition}

\begin{proof}
Since $\phi(N)=N\prod_{p|N} \left(1-\frac 1 p\right)$,
\[
\Phi_N(x)=\frac {N} {\varphi(N)} |x|\sum_{\substack{
n\in N\NN,\\
n\ge |x|}} \frac {\varphi(n)}{|n|^3}
=\prod_{p|N}\left(1-\frac 1 p\right)^{-1} |x|\sum_{\substack{
n\in N\NN,\\
n\ge |x|}} \frac {\varphi(n)}{|n|^3}.
\]
Hence it suffices to show that
\[
|x|\sum_{\substack{
n\in N\NN,\\
n\ge |x|}} \frac {\varphi(n)}{|n|^3}=\frac 1 {\zeta_N(2)N} \prod_{p|N} \left(1-\frac 1 p\right) + O(|x|^{-1}\log|x|).
\]
This is a direct consequence of Abel's summation theorem and Lemma~\ref{lem: estimation 1}.
\end{proof}

\begin{lemma}\label{lem: estimation 1} For any $K>1$,
\[
\sum_{\substack{
1\le n \le K,\\
n\in N\NN}} \varphi(m)=\left(\frac 1 {\zeta_N(2)N}\prod_{p|N} \left( 1- \frac 1 p\right)\right)\frac {K^2} 2 +O_N(K\log K).
\]
\end{lemma}
\begin{proof}
Let us focus on how one can reserve the leading term. The error term for each step is easily obtained in a similar way with the proof of \cite[Lemma 4.1]{FH}.

Using the fact that $\varphi(m)=m\sum_{d|m} \mu(d)/d$ for any $m\in \NN$, where $\mu$ is the M\"obius function,
\[
\sum_{\substack{
1\le m \le K,\\
m\in N\NN}} \varphi(m)
=\sum_{\substack{
1\le m\le K,\\
m\in N\NN}} m \sum_{d|m} \frac {\mu(d)} {d}
=\sum_{1\le d\le K} \mu(d) \sum_{\substack{
1\le d'\le K/d,\\
dd'\in N\NN}} d'.
\]
Since $d'd\in N\NN$ and $d'\in \NN$ implies $d'\in \frac N{\gcd(N,d)}\NN$, 
\[\begin{split}
\sum_{\substack{
1\le d'\le K/d,\\
dd'\in N\NN}} d'
&=\sum_{k'=1}^{\lfloor K\gcd(N,d)/(dN)\rfloor} \frac N {\gcd(N,d)} k'\\
&=\frac N {\gcd(N,d)} \frac 1 2 \left(\left\lfloor \frac {K\gcd(N,d)} {Nd} \right\rfloor^2+\left\lfloor \frac {K\gcd(N,d)} {Nd} \right\rfloor\right),
\end{split}\]
where $\lfloor x\rfloor$ is the largest integer less than or equal to $x\in \RR$. 
Putting $d_1=\gcd(N,d)$, we obtain
\[\begin{split}
\sum_{\substack{
1\le m\le K,\\
m\in N\NN}} \varphi(m)
&=\sum_{1\le d \le K} \mu(d) \frac {K^2\gcd(N,d)} {2d^2N} + O(K\log K)\\
&=\frac {K^2} {2N} \sum_{d_1|N} \sum_{\substack{
1\le d\le K,\\
\gcd(d,N)=d_1}} \frac {\mu(d)} {d^2} d_1 + O(K\log K).
\end{split}\]

Now we claim that
\[
\sum_{\substack{
1\le d\le K,\\
\gcd(N,d)=d_1}} \frac {\mu(d)}{d^2}\: d_1
=\sum_{\substack{
1\le m \le K/d_1,\\
\gcd(m,N/d_1)=1}} \frac {\mu(d_1m)} {(d_1m)^2} \:d_1
=\frac {\mu(d_1)}{d_1} \sum_{\substack{
1\le m \le K/d_1,\\
\gcd(m,N)=1}} \frac {\mu(m)} {m^2}.
\]
Indeed, $\mu(d_1m)\neq 0$ only if $\gcd(m,N)=1$.
In this case, $\gcd(m,N/d_1)=1$ if and only if $\gcd(m, N)=1$. Moreover, it holds that $\mu(d_1m)=\mu(d_1)\mu(m)$.

According to the proof of \cite[Lemma 4.1]{FH} (by taking $S=\{\infty\}\cup\{\text{primes dividing }N\}$), it follows that for any $d_1$,
\[
\sum_{\substack{
1\le m\le K/d_1,\\
\gcd(m,N)=1}} \frac {\mu(m)} {m^2}
=\sum_{\substack{
m\in \NN,\\
\gcd(m,N)=1}}\frac {\mu(m)} {m^2} + O(K\log K)
=\frac 1 {\zeta_N(2)} + O(K\log K).
\]
Therefore, we have
\[\begin{split}
\sum_{\substack{
1\le m\le K,\\
m\in N\NN}} \varphi(m)
&=\frac {K^2} {2N} \sum_{d_1|N} \frac {\mu(d_1)}{d_1}\left(\frac 1 {\zeta_N(2)} +O(K\log K)\right)\\
&=\frac 1 {\zeta_N(2)N} \prod_{p|N} \left(1-\frac 1 p\right)\cdot \frac {K^2} 2 + O(K\log K).
\end{split}\]
\end{proof}
\section{Proofs of Counting Results}
As stated in the introduction, the main steps of the proof of Theorem~\ref{thm: Analog of Schmidt} are essentially parallel to those in Schmidt's original work \cite{Schmidt1960}, thus we do not reproduce the full argument here. We instead provide a sketch of the proof, focusing on the components that interact with the moment formulas presented earlier.

\begin{proof}[Sketch of the proof of Theorem~\ref{thm: Analog of Schmidt}.]
Let $\{A_T\}_{T>0}$ be an increasing family of Borel sets in $\RR^2$ with $V_T=\vol(A_T)\rightarrow \infty$ as $T\rightarrow \infty$. By re-indexing if necessary, we may further assume that $V_T=T$ for any $T>0$.

For each $m\in \NN$, define
\[
K_{m}=\left\{(M_1, M_2): 0\le M_1<M_2\le 2^m, \; M_1= u2^\ell,\; M_2=(u+1)2^{\ell}\;\text{for some}\; u,\ell\in \ZZ_{\ge 0} \right\}.
\]
We apply Borel--Cantelli lemma on the collection $\{\mathfrak B_m\}_{m\in \NN}$, where
\[
\mathfrak B_m=\left\{t^{1/2}g\in \mathfrak C: \left|t\widehat{\mathbbm 1}^{(\sigma)}_{A_M}(t^{1/2}g) - \frac {V_M}{\zeta_N(2)N^2}\right|\ge m^2 2^{m/2} \psi^{1/2}(m\log2 -1),\;\forall M\le 2^m \right\}.
\]

Indeed, one can show that
\begin{equation*}
\int_{\mathfrak B_m} 1 \:d \nu_N(g)dt\le \frac {c} {\psi(M\log2 -1)} 
\end{equation*}
for some $c>0$ (see \cite[Lemma 13]{Schmidt1960}), thus the summation of measures of $\mathfrak B_m$ converges under our assumption that $\int \psi^{-1}(x) dx<\infty$.
And one can obtain the above inequality directly from the following: there is a constant $c_1>0$ so that
\begin{equation}\label{eqn 1: sketch}
\sum_{(M_1, M_2)\in K_m} \int_{\mathfrak C} \left(t\widehat{\mathbbm 1}^{(\sigma)}_{A_{M_2}-A_{M_1}}(t^{1/2}g) - \frac {V_{M_2}-V_{M_1}} {\zeta_N(2)N^2}\right)^2 d\nu_N(g) dt< c_1m^3 2^m,
\end{equation}
together with observing that for any $M\le 2^m$, the interval $[0, M)$ can be expressed as the union of at most $2m$ intervals of the type $[M_1, M_2)$, where $(M_1, M_2)\in K_m$, and Cauchy--Schwarz inequality (\cite[Lemma12]{Schmidt1960}). Note that the function $\mathbbm 1_{A_M}$ can be written by a sum of at most $2M$ terms of the form $\mathbbm 1_{A_{M_2}-A_{M_1}}$, and likewise for $t\widehat{\mathbbm 1}^{(\sigma)}_{A_{M}}$ with terms $t\widehat{\mathbbm 1}^{(\sigma)}_{A_{M_2}-A_{M_1}}$.

Combining first moment formula (Theorem~\ref{thm: First moment on C}) and the second moment formula (Theorem~\ref{thm: Second moment on C}), it follows that for any $A\subseteq \RR^2$ (we will put $A=A_{M_2}-A_{M_1}$),
\[\begin{split}
\int_{\mathfrak C} \left(t\widehat{\mathbbm 1}^{(\sigma)}_A (t^{1/2}g) - \frac {\vol(A)}{\zeta_N(2)N^2}\right)^2 d\nu_N(g)dt
&\le
\frac 1 {2\zeta_N(2)N^2} \int_{\RR^2} \mathbbm 1_A(\bf{v})\left(\mathbbm 1_{A}(\bf{v})+\mathbbm 1_{A}(-\bf{v})\right) d\bf{v}\\
&\hspace{-1.4in}+\frac 1 {\zeta_N(2)N^3} \int_{\RR^2\times \RR^2} \left(\Phi_N(\det(\bf{v}_1, \bf{v}_2)) -\frac 1 {\zeta_N(2)N}\right)\mathbbm 1_{A}(\bf{v}_1) \mathbbm 1_{A}(\bf{v}_2) d\bf{v}_1 d\bf{v}_2.
\end{split}\]

Proposition~\ref{Prop: asym formula of Phi_N} yields that there are constants $c_2, \;c_3$ and $V_0>0$ such that if $V=\vol(A)>V_0$, one can take a non-negative and non-increasing function $\chi$ on $\RR_{>0}$ 
\[
\chi(x)=\left\{\begin{array}{cl}
c_2&\text{if}\; x<V_0/\log V_0;\\[0.05in]
c_3x^{-1}\log x&\text{if}\; V_0/\log V_0 \le x < V/\log V;\\[0.05in]
0&\text{if}\; x\ge V/\log V\\
\end{array}\right.\]
so that the following holds: $\int_{0}^\infty \chi(x)dx \le c_4 (\log V)^2$ for some $c_4>0$, thus there is $c_5>0$ such that 
\[
\frac 1 {\zeta_N(2)N^3} \int_{\RR^2\times \RR^2} \left(\Phi_N(\det(\bf{v}_1, \bf{v}_2)) -\frac 1 {\zeta_N(2)N}\right)\mathbbm 1_{A}(\bf{v}_1) \mathbbm 1_{A}(\bf{v}_2) d\bf{v}_1 d\bf{v}_2 \le c_5 V(\log V)^2
\]
(see the proof of \cite[Lemma 10 and Lemma 11]{Schmidt1960}).

Notice that the function $\chi$ depends on the Borel set $A$, specifically on the volume of $A$. 
Here, we use \cite[Theorem 3]{Schmidt1960} which states when $d=2$ that
\[
\int_{\RR^2\times \RR^2} \mathbbm 1_A (\bf{v}_1) \mathbbm 1_A(\bf{v}_2) \chi(|\det(\bf{v}_1, \bf{v}_2)|)
d\bf{v}_1 d\bf{v}_2 \le 8 \vol(A)\int_0^\infty \chi(x) dx.
\]

Hence, we obtain that
\[
\int_{\mathfrak C} \left(t\widehat{\mathbbm 1}^{(\sigma)}_A (t^{1/2}g) - \frac {\vol(A)}{\zeta_N(2)N^2}\right)^2 d\nu_N(g)dt \le c_6 V(\log V)^2.
\]
The upperbound \eqref{eqn 1: sketch} is deduced by summing all integrals for the sets $A=A_{M_2}-A_{M_1}$, where $(M_1, M_2)\in K_m$. Note that we use a decomposition of $K_m$, with possible repetitions, into the $m$-number of chains of pairs of $(M_1, M_2)$ as follows:
\[\begin{array}{c}
(0\cdot 2, 1\cdot 2), (1\cdot 2, 2\cdot 2), (2\cdot 2, 3\cdot 2), ..., ((2^{m-1}-1)\cdot 2, 2^{m-1}\cdot 2);\\
(0\cdot 2^2, 1\cdot 2^2), (1\cdot 2^2, 2\cdot 2^2), (2\cdot 2^2, 3\cdot 2^2), ..., ((2^{m-2}-1)\cdot 2^2, 2^{m-2}\cdot 2^2);\\
\vdots\\
(0\cdot 2^{m-1}, 1\cdot 2^{m-1}), (1\cdot 2^{m-1}, 2\cdot 2^{m-1}).
\end{array}\]
The partial sum over each chain in \eqref{eqn 1: sketch} is bounded by $c_6\vol(A_M)\log(\vol(A_M))^2\le c_6 2^{m} m^2$. 
\end{proof}

\begin{proof}[Proof of Theorem~\ref{thm: quantitative K}.] The key idea is to use Theorem~\ref{thm: Analog of Schmidt} with the family of Borel sets 
\[
\left\{A_T=\left\{(x,y)\in \RR^2: |x|\le \phi(|y|)\;\text{and}\; |y|<T\right\}\right\}_{T\in \RR_{>0}}.
\]
Moreover, we want to put $g=u_x=\left(\begin{array}{cc} 1 & x \\ 0 & 1\end{array}\right)$ for generic $x\in \RR$. For this reduction, we consider the small neighborhood of $\{u_x\in \SL_2(\RR): x\in \RR\}$ by thickening in direction 
\[\left\{\left(\begin{array}{cc} y & 0 \\ z & y^{-1}\end{array}\right): y\neq 0, \;z\in \RR \right\}.
\]

For the details, we refer the reader to \cite{AGY2021}. While their result is stated for $d\ge3$, due to the lack of moment formulas in dimension two, the proof of \cite[Theorem 1]{AGY2021} that we require here is not directly affected by this limitation.
\end{proof}



\end{document}